\documentclass [12pt,a4paper]{article}
\usepackage{amsfonts}\usepackage{graphicx}
\usepackage{amsmath}
\usepackage{amsthm}
\usepackage{amssymb}
\usepackage{amscd}
\usepackage{amsbsy}

\advance\textwidth by +1.0in \advance\textheight by +1.0in
\advance\oddsidemargin by -0.5in \advance\evensidemargin by -1.0in
\advance\topmargin by -0.5in
\bibliography{bibfile}

\newtheorem {theorem} {Theorem}
\newtheorem {proposition} [theorem]{Proposition}

\newtheorem {lemma}  [theorem]{Lemma}

\def\la{\lambda}

\newcommand{\bbox}{\ \hfill\rule[-1mm]{2mm}{3.2mm}}

\title{\large\bf The embedding flows of $C^\infty$ hyperbolic diffeomorphisms}
\author{\normalsize\bf\sc Xiang Zhang\footnote{\small
The author is partially supported by
NNSF of China grants 10671123 and 10831003, and the Shanghai Pujiang Program 09PJD013.}\\
\normalsize\it Department of Mathematics, Shanghai Jiaotong
               University,  \\ \normalsize\it Shanghai 200240,
               The People's Republic of China.
      \\ \normalsize  E-mail: xzhang@sjtu.edu.cn\\}
\date{}
\begin{document}
\maketitle
\begin{abstract}
\noindent  In [{\it American J. Mathematics}, 124(2002),
107--127] we proved that  for a germ of  $C^\infty$
hyperbolic diffeomorphisms $F(x)=Ax+f(x)$ in $(\mathbb R^n,0)$, if
$A$ has a real logarithm with its eigenvalues weakly nonresonant,
then $F(x)$ can be embedded in a $C^\infty$ autonomous differential system.
Its proof was very complicated, which involved the existence of embedding periodic
vector field of $F(x)$ and the extension of the Floquet's theory to nonlinear $C^\infty$ periodic
differential systems. In this paper we shall provide a simple and
direct proof to this last result.

Next we shall show that the weakly nonresonant condition in the last result
on the real logarithm of $A$ is necessary for some $C^\infty$ diffeomorphisms $F(x)=Ax+f(x)$ to
have $C^\infty$ embedding flows. 

Finally we shall prove that a germ of $C^\infty$ hyperbolic diffeomorphisms $F(x)=Ax+f(x)$ with $f(x)=O(|x|^2)$
in $(\mathbb R^2,0)$ has a $C^\infty$ embedding flow if and only if
either $A$ has no negative eigenvalues or $A$ has two equal negative eigenvalues and it can be diagonalizable.

\hskip0.06mm


\noindent {\bf 2000 AMS mathematical subject classification:}
34A34, 34C20, 34C41, 37G05.
\end{abstract}

\bigskip
\section{Introduction and statement of
the main results}

\setcounter{section}{1}
\setcounter{equation}{0}\setcounter{theorem}{0}

 \noindent For a
germ of $C^k$ smooth diffeomorphims in $(\mathbb R^n,0)$ with
$k\in\mathbb N\cup\{\infty,\omega\}$, a vector field $\mathcal X$
defined in $(\mathbb R^n,0)$ is an {\it embedding vector field} of
${F}(x)$ if $ F( x)$ is the time one map of the flow induced by
$\mathcal X$. The flow of $\mathcal X$ is called an {\it embedding
flow} of $F(x)$. For a vector field $\mathcal Y$ depending on the time
(if it is periodic in the time, we assume that its period is $1$), if $F(x)$ is the
time one map of its solutions, then $\mathcal Y$ is also called an 
{\it embedding vector field} of $F(x)$.

The embedding flow problem is classical, which asks to solve the existence and
smoothness of embedding flows for a given diffeomorphism with some
smoothness. We should say that the embedding flow problem appears
naturally when we study the relation between diffeomorphisms and
vector fields.  Recently the results on the existence of embedding flows were successfully applied to study
the integrability, especially the inverse integrating factor, of planar
differential systems (see e.g. \cite{EP09,GGG09,GM09}).

For 1--dimensional diffeomorphisms the embedding flow problem has
been intensively studied (see for instance,
\cite{Be,BC,La1,La2,Li}). For higher dimensional diffeomorphisms, Palis
\cite{Pa} in 1974 pointed out that the diffeomorphisms admitting
embedding flows are few in the Baire sense.

In 1988 Arnold \cite[p.200]{Ar} mentioned that any $C^\infty$
local diffeomorphism  $F (x)= Ax + o(|x|)$ in $(\mathbb R^n, 0)$
with $A$ having a real logarithm admits a $C^\infty$ periodic
embedding vector field in $(\mathbb R^n, 0)$. In 1992 Kuksin and
P\"{o}schel \cite{KP} solved the problem on embedding a $C^\infty$
smooth or an analytic symplectic diffeomorphism into a $C^\infty$
smooth or an analytic periodic Hamiltonian vector field in a
neighborhood of the origin.

The last two results proved the existence of periodic embedding
vector fields for a given diffeomorphism. Arnold \cite[p.200]{Ar} 
mentioned that it is usually not possible to embed a given mapping
in the phase flow of an autonomous system.  This implies that the
embedding flow problem is much more involved.

In the past decade there are some progresses on the embedding flow problem. 
For recalling these results we need some definitions.

For a local diffeomorphism $F(x)=Ax+f(x)$ in $(\mathbb
R^n,0)$, the eigenvalues $\lambda=(\lambda_1,\ldots,\lambda_n)$ of
$A$ is {\it resonant} if there exists some $k=(k_1,\ldots,k_n)\in\mathbb Z_+^n$
with $|k|\ge 2$ such that $\lambda_j=\lambda^k$ for
some $j\in\{1,\ldots,n\}$, where $\mathbb Z_+$ denotes the set of
nonnegative integers, and as usual $|k|=k_1+\ldots+k_n$ and  $\lambda^k=\lambda_1^{k_1}\ldots\lambda_n^{k_n}$.

For a vector field $\mathcal X(x)=Bx+v(x)$, the eigenvalues $\mu=(\mu_1,\ldots,\mu_n)$ of
$B$ is
\begin{itemize}
\item{} {\it resonant} if there exists some $k=(k_1,\ldots,k_n)\in\mathbb Z_+^n$
with $|k|\ge 2$ such that $\mu_j=\langle k,\mu\rangle$ for
some $j\in\{1,\ldots,n\}$, where $\langle k,\mu\rangle=k_1\mu_1+\ldots+k_n\mu_n$;
\item{} {\it weakly resonant} if there exists some $k\in\mathbb Z_+^n$
with $|k|\ge 2$ such that $\mu_j-\langle k,\mu\rangle=2l\pi\sqrt{-1}$ for
some $j\in\{1,\ldots,n\}$ and $l\in\mathbb Z\setminus \{0\}$.
\end{itemize}

Now we recall some recent results on the embedding flow problem. In 2002 Li, Llibre and Zhang \cite{LLZ} solved the
embedding flow problem for a class of $C^\infty$ hyperbolic diffeomorphisms, stated 
as follows.
\begin{theorem} \label{th1} For any $C^\infty$ locally hyperbolic
diffeomorphism ${ F}(x)= A x+f(x)$ with ${f}(x)=O(|x|^2)$ in
$(\mathbb R^n,0)$, if ${ A}$ has a real logarithm $B$ with its
eigenvalues weakly nonresonant, then the diffeomorphism $
F(x)$ admits a $C^\infty$ embedding flow induced by a $C^\infty$ vector field of
the form $Bx+v(x)$ with $v(x)=O(|x|^2)$.
\end{theorem}

 This last result was recently extended to the $C^\infty$
local diffeomorphisms in the Banach spaces \cite{Zh} by using the spectral theory
of linear operators and using some
ideas from \cite{LLZ}. In 2008 we solved the embedding flow
problem for the volume--preserving analytically integrable
diffeomorphisms \cite{Zh1}. For the finite smooth diffeomorphisms some
partial results on this problem were recently
obtained in \cite{Zh2}.

In \cite{LLZ}, for proving Theorem \ref{th1} we first proved that
the given diffeomorphism $F(x)$ can be embedded in a $C^\infty$
periodic vector field (in fact, it is a result mentioned by Arnold in \cite{Ar}
without a proof), then we proved that the periodic vector field
is $C^\infty$ equivalent to an autonomous vector field by using the
extended Floquet's theory for nonlinear periodic differential
vector field, which is also proved in \cite{LLZ}.
The proof was very complicated and involved.
In this paper we shall provide a simple and direct proof to
Theorem \ref{th1}, i.e. without using the embedding periodic vector fields
and so it is not necessary to use the extended Floquet's theory, see Subsection \ref{s22}.

In Theorem \ref{th1} there is an additional condition: weakly nonresonance
on the eigenvalues of the real logarithm
$B$ of $A$. This condition was used in \cite{LLZ} for reducing a $C^\infty$ nonlinear periodic
differential system to an autonomous one. For a long time we do not know if this condition is necessary or not for the
existence of the $C^\infty$ embedding flows of a $C^\infty$ hyperbolic diffeomorphism. Our next result shows that this condition is necessary for some
$C^\infty$ (resp. analytic) diffeomorphism to have a $C^\infty$ (resp. an analytic) embedding flow.

\begin{theorem}\label{ex1}
For $n\ge 3$, there exist locally $C^\infty$ $($resp. analytic$)$ diffeomorphisms of the form $F(x)=Ax+f(x)$ in $(\mathbb R^n,0)$
which have no $C^\infty$ $($resp. analytic$)$ embedding autonomous vector fields, where  $f(x)=O(|x|^2)$ and $A$ has a real logarithm with its eigenvalues weakly resonant.
\end{theorem}

The proof of Theorem \ref{ex1} will be given in Subsection \ref{s23}. At the end of Subsection \ref{s22}
we shall explain the importance of weakly nonresonance to the existence of embedding flows.

Let $B$ be a real logarithm of $A$ and let $\lambda=(\lambda_1,\ldots,\lambda_n)$ and $\mu=(\mu_1,\ldots,\mu_n)$ be the $n$--tuples
of eigenvalues of $A$ and $B$, respectively. Then $\lambda_j=e^{\mu_j}$ (by a permutation if necessary)
for $j=1,\ldots,n$. Associated with Theorem \ref{ex1} and its proof, we pose the following

\noindent{\bf Conjecture.}
{\it If $f(x)=O(|x|^2)$ is $C^\infty$ $($resp. analytic$)$ and its $j^{th}$ component has a resonant monomial $x^me_j$ with $m$ satisfying
$\mu_j-\langle m,\mu\rangle=2k\pi\sqrt{-1}$ for some $k\in\mathbb Z\setminus\{0\}$, then the locally hyperbolic
diffeomorphism $F(x)=Ax+f(x)$
has no $C^\infty$ $($resp. analytic$)$ embedding flows.
}

Recall that $e_j$ is the $j^{th}$ unit vector with its $j^{th}$ entry being equal to $1$ and the others all vanishing.

Using Theorem \ref{th1} and the following Proposition \ref{po1} we can get the following

\begin{theorem} \label{co1} For a locally hyperbolic $C^\infty$
diffeomorphism ${ F}(x)= A x+f(x)$ with ${f}(x)=O(|x|^2)$ in
$(\mathbb R^n,0)$, the following statements hold.
\begin{itemize}
\item[$(a)$] For $n=2$, $F(x)$ has a $C^\infty$ embedding flow if and only if
either $A$ has no negative eigenvalues, or $A$ has the Jordan normal form $\mbox{\rm diag}(\lambda,\lambda)$ with
$\lambda<0$.
\item[$(b)$] For $n\ge 3$, if $A$ has eigenvalues all positive,
then the diffeomorphism $
F(x)$ admits a $C^\infty$ embedding flow induced by the vector field of
the form $Bx+v(x)$ with $v(x)=O(|x|^2)$.
\end{itemize}
\end{theorem}

We note that statement $(a)$ solves completely the embedding flow problem for the 
planar locally hyperbolic $C^\infty$ diffeomorphisms.

Statement $(b)$ is in fact Corollary 5 of \cite{LLZ} without a proof there. We list and prove it
here for completeness. Contrary to
statement $(b)$,  if $A$
has negative or conjugate  complex eigenvalues, there are hyperbolic $C^\infty$ diffeomorphisms of the form
$F(x)=Ax+f(x)$ which have no $C^\infty$ embedding flows. For example, the local diffeomorphims in $(\mathbb R^3,0)$
\[
F_1(x,y,z)=(-2x,-2y,4z+axy+bx^2+cy^2),
\]
with its linear part having a real logarithm
\[
B=\left(\begin{array}{ccc}\ln 2 & \pi & 0\\ -\pi & \ln 2 & 0\\ 0 & 0 & 2\ln 2\end{array}
\right),
\]
have neither $C^\infty$ nor analytic embedding autonomous vector field provided that $a^2+b^2+c^2\ne 0$ and $b\ne c$. Its
proof follows from the same lines as that of Theorem \ref{ex1}, the details are omitted. We can check easily that in this last
example the eigenvalues of $B$ are weakly resonant.
From the proof of Theorem \ref{ex1} we can subtract an example of $C^\infty$ local diffeomorphisms with their
linear parts having conjugate complex eigenvalues and having a real logarithm which has no embedding
autonomous vector fields.

Finally we make a remark on the main result of \cite{Zh2}.
According to Theorem \ref{ex1} and its proof we should add the weakly nonresonant condition to
Theorem 1.1 of \cite{Zh2} for ensuring the
existence of embedding flows. We restate it here with the additional condition on weakly nonresonance.

\begin{theorem} \label{th1.1} Assume that ${\bf F}(x)=\mathbf A x+\mathbf f(x)$
with ${\bf f}(x)$ a resonant vector--valued polynomial without
linear part and that ${\bf A}$ is hyperbolic and has a real
logarithm ${\bf B}$ with its eigenvalues weakly nonresonant. If the eigenvalues $\mathbf \la$ of $\mathbf
A$ satisfy finite resonant conditions, then ${\bf F}(x)$ has a unique
embedding vector field. Furthermore, the embedding  vector field is of the form ${\mathcal X(x)}={\bf B}x+{\bf v}(x)$ with
${\bf v}(x)$ a resonant polynomial vector field.
\end{theorem}

Having the additional weakly nonresonant condition, the proof of Theorem \ref{th1.1} follows from the same line as that of
Theorem 1.1 of \cite{Zh2}. In fact, the proof of Theorem \ref{th1.1} can also be completed by using
that of Lemma \ref{le6}.

At the end of this section we pose the following

\noindent{\bf Open problem.}
{\it What kinds of $C^\infty$ non--hyperbolic
local diffeomorphisms can be embedded in a $C^\infty$ flow? }

We should mention that this last open problem was solved recently in
\cite{Zh3} for smooth and analytic integrable diffeomorphisms.
In non--integrable case there are no any results except those local
diffeomorphisms which can be linearized by a $C^\infty$ near identity transformation.

This paper is organized as follows. In the next section we shall prove our main results. Subsection
\ref{s21} is the preparation to the proof of the main results,
which contains some necessary known results. Subsection \ref{s22}
is the new and direct approach to the proof of Theorem \ref{th1}. Subsections
\ref{s23} and \ref{s24} are the proofs of  Theorem \ref{ex1} and of Theorem \ref{co1}, respectively.

\section{Proof of the main results}\label{s2}

\setcounter{section}{2}
\setcounter{equation}{0}\setcounter{theorem}{0}

\subsection{Preparation to the proof}\label{s21}

In this subsection we present some known results, which will be
used in the proof of our main results.

The first one gives the relation between the embedding vector
fields of two conjugate diffeomorphisms (see e.g. \cite{LLZ}). Its
proof is easy.

\begin{lemma}\label{le1} Let ${ G}$ and ${ H}$ be two conjugate
diffeomorphisms on a manifold, i.e. there exists a diffeomorphism
${J}$ such that ${J}\circ { G}={ H}\circ { J}$. If ${\mathcal X}$
is an embedding vector field of ${ G}$, then ${ J}_*{\mathcal X}$
is an embedding vector field of ${H}$, where ${ J}_*$ denotes the
tangent map induced by the conjugation ${J}$.
\end{lemma}

This lemma shows that in order for solving the embedding flow problem for a
given diffeomorphism, one way is to find a conjugate
diffeomorphism which has an embedding flow.

The next result, due to Chen \cite{Ch}, characterizes the
equivalence between $C^\infty$ conjugacy and formal conjugacy,
which is the key point to prove our main result.

\begin{theorem}\label{le2}
If two $C^\infty$ diffeomorphisms are formally conjugate at a
hyperbolic fixed point, then they are $C^\infty$ conjugate in a
neighborhood of the fixed point.
\end{theorem}

The following one, called the Poincar\'e--Dulac theorem, provides the existence of the 
analytic normal form for an analytic diffeomorphism (see for instance
\cite{AA,Li}), which will also be used in the proof of our main results.
\begin{theorem}\label{le7}
Any local analytic diffeomorphism $F(x)=Ax+f(x)$ in $(\mathbb C^n,0)$, with $f(x)=o(x)$ and
the eigenvalues of $A$ belonging to the Poincar\'e domain, is analytically conjugate to
its normal form.
\end{theorem}

Recall that the eigenvalues of $A$ belong to the {\it Poinar\'e domain}
if their modulus are either all larger than $1$ or all less than $1$.

The following result presents the spectrum of a class of linear
operators associated with maps (for a proof, see e.g. \cite{Zh2})

\begin{lemma}\label{le3} Let $M_n(\mathbb F)$ $(\mathbb F=\mathbb C$ or $\mathbb R)$ be the set of square
matrices of order $n$ with its entries being in the field $\mathbb F$, and
let ${\cal H}_n^r(\mathbb F)$ be the linear space formed by 
$n$--dimensional vector--valued homogeneous polynomials of degree
$r$ in $n$ variables with their coefficients in $\mathbb F$. If ${A}\in
M_n(\mathbb F)$ is a lower triangular Jordan normal form matrix
with the eigenvalues $\la_1,\ldots,\la_n$, then the linear operator in
${\bf\cal H}_n^r(\mathbb F)$ defined by
\begin{equation}\label{e1}
\mathcal L{h}(x)={ Ah}(x)-{h}({ A}x),\qquad { h}\in {\cal
H}_n^r(\mathbb F),
\end{equation}
has the spectrum
\[
\left\{\la_j-\prod\limits_{i=1}\limits^{n}\la_i^{m_i};\,
m=(m_1,\ldots,m_n)\in{\mathbb Z}_+^n,\,|m|=r,\, j=1,\ldots,n\right
\},
\]
where $\mathbb Z_+$ denotes the set of nonnegative integers and
$|m|=m_1+\ldots+m_n$.
\end{lemma}

The next result, due to Bibikov \cite{Bi}, provides the spectrum
of a class of linear operators associated with the vector fields.

\begin{lemma}\label{le3.1} For ${B}\in M_n(\mathbb F)$, we define a linear operator $L$ in ${\mathcal
H}_n^r(\mathbb F)$ by
\[
L{h}(x)=D{h}(x){B}x-{Bh}(x),\qquad {h}\in {\mathcal H}_n^r(\mathbb
F).
\]
If $\mu_1,\ldots,\mu_n$ are the eigenvalues of $B$, then the
spectrum of $L$ is
\[
\left\{\sum\limits_{i=1}\limits^{n}m_i\mu_i-\mu_j;\,
m=(m_1,\ldots,m_n)\in{\mathbb Z}_+^n,\,|m| =r,\,
j=1,\ldots,n\right \}.
\]
\end{lemma}

The next lemma will be used in the proof of the following Lemmas
\ref{le4} and \ref{le6} (for a proof, see for instance
\cite{Li,LLZ}).

\begin{lemma}\label{lerr}
For any sequence $\{a_m\in \mathbb R^n;\,\,
m=(m_1,\ldots,m_n)\in\mathbb Z_+^n\}$, there exists a $C^\infty$
function $q(y)$ such that
\[
\left.\frac{\partial^{|m|} q(y)}{\partial y_1^{m_1}\ldots
\partial y_n^{m_n}}\right|_{y=0}=a_m\qquad \mbox{ for all } m=(m_1,\ldots,m_n)\in\mathbb
Z_+^n.
\]
\end{lemma}

The last one, due to Culver \cite{Cu},
provides the necessary and sufficient conditions for a real square
matrix to have a real logarithm (for a different proof, see example  Li, Llibre and Zhang \cite{LLZ}).
\begin{proposition}\label{po1}
A nonsingular real square matrix $A$ has a real logarithm if and
only if either $A$ has no negative real eigenvalues, or the Jordan
blocks in the Jordan normal form of $A$ corresponding to the negative
real eigenvalues appear pairwise, i.e. there is an even number of
such blocks: $J_1,\ldots,J_{2m}$ with $J_{2i-1}=J_{2i}$ for
$i=1,\ldots,m$.
\end{proposition}

\subsection{\normalsize A simple and direct Proof to Theorem \ref{th1}}\label{s22}

For any invertible real square matrix $C$ of order $n$,
the function $C^{-1}ACx+C^{-1}f(Cx)$ is analytically conjugate to
$F(x)=Ax+f(x)$. By Lemma \ref{le1} these two diffeomorphisms
either both have embedding flows with the same regularity or both
have no embedding flows. So without loss of generality we assume that $A$ is
in the real normal form, i.e.
\[
A=\mbox{diag}\,(A_1,\ldots,A_k,B_1,\ldots,B_l),
\]
with
\[
A_r=\left(\begin{array}{cccc}\lambda_r & & & \\1 & \lambda_r & & \\ & \ddots &\ddots &\\ & & 1 & \lambda_r\end{array}\right),\, r\in\{1,\ldots,k\},\,\,\,
B_s=\left(\begin{array}{cccc}D_s & & & \\E_2 & D_s & & \\ & \ddots &\ddots &\\ & & E_2 & D_s\end{array}\right), \,s\in\{1,\ldots,l\},
\]
where $\lambda_r\in \mathbb R$ and
\[
D_s=\left(\begin{array}{cc}\alpha_s & \beta_s \\-\beta_s & \alpha_s\end{array}\right),\quad
E_2=\left(\begin{array}{cc}1 & 0 \\ 0 & 1\end{array}\right).
\]
Assume that $m$ is the total number of the real eigenvalues of $A$ taking into account their multiplicity, and that $B_s$ is of order $2 m_s$. For $s=1,\ldots,l$, associated to the conjugate complex eigenvalues $\alpha_s\pm \sqrt{-1}\beta_s$ and the corresponding real coordinates $(x_i,x_{i+1})$ with $i=m+1+2(m_{0}+\ldots+m_{s-1})+2j$ for $s=1,\ldots,l;\,j=0,\ldots,m_s-1$ (where $m_0=0$), we take $z_i=x_i+\sqrt{-1}\,x_{i+1}$ and set
\[
F^*_i=F_i[x]
+\sqrt{-1}\, F_{i+1}[x],\quad F^*_{i+1}=F_i[x]-\sqrt{-1}\, F_{i+1}[x],
\]
where $F_i$ is the $i^{th}$ component of $F$ and $[x]$ is $(x)$ with $x_j$ for $j\in\{m+1,\ldots,n\}$ replaced by the conjugate complex coordinates $z_j$ and $\overline z_j$ through $x_j=(z_j+\overline z_j)/2$ and $x_{j+1}=(z_j-\overline z_j)/(2\sqrt{-1})$. Then $F^*=(F_1,\ldots,F_m,F^*_{m+1},\ldots,F^*_n)$ has its linear part in the lower triangular normal form.
Furthermore if $F^*$ has an embedding autonomous differential system
\begin{eqnarray*}
\dot x_i&=&g_i,\quad i=1,\ldots,m,\\
\dot z_j&=&g_j,\quad j=m+1,m+3,\ldots,n-1,\\
\dot {\overline z}_j&=&\overline {g}_j,
\end{eqnarray*}
then written back in the real coordinates the real autonomous system
\begin{eqnarray*}
\dot x_i&=&g_i,\qquad\,\,\, i=1,\ldots,m,\\
\dot x_j&=&\mbox{Re}\,g_j,\quad j=m+1,m+3,\ldots,n-1,\\
\dot x_{j+1}&=&\mbox{Im}\, g_j,
\end{eqnarray*}
is an embedding differential system of $F$. Hence in what follows we only need to prove the existence of the embedding autonomous differential system for $F^*$. For simplicity to notations, we still study $F(x)$ instead of $F^*$ and assume that $F(x)$ has $A$ in the lower triangular Jordan normal form.

In the proof of our main results, we need to use the normal forms of a given
diffeomorphism. First we give some definitions. A diffeomorphism $H(x)=Ax+h(x)$ is in the {\it normal
form} if its nonlinear term $h(x)$ contains only resonant
monomials. A monomial $a_mx^m e_j$ in the $j^{th}$ component of $h(x)$ is {\it resonant} if
$\lambda^m=\lambda_j$, where $m\in\mathbb Z_+^n$ and $|m|\ge 2$,
$\lambda=(\lambda_1,\ldots,\lambda_n)$ are the $n$--tuple of
eigenvalues of the matrix $A$. Recall that $e_j$ is the unit vector in $\mathbb R^n$ with its
$j^{th}$ entry being equal to one and the others vanishing. Recall that 
$\lambda^m=\lambda_1^{m_1}\ldots\lambda_n^{m_n}$ and $x^m=x_1^{m_1}\ldots x_n^{m_n}$ for
$m=(m_1,\ldots, m_n)$ and $x=(x_1,\ldots,x_n)$.

For a near identity transformation
$x=y+\zeta(y)$ with $\zeta(y)=O(|y|^2)$ from $H(x)$ to its normal
form $G(y)$, if $\zeta(y)$ contains only nonresonant term, then it
is called {\it distinguished normalization} from $H(x)$ to $G(y)$.
Correspondingly, $G(y)$ is called {\it distinguished normal form}
of $H(x)$. Recall that for a given diffeomorphism, the asymptotic development of its distinguished normal form is unique.
Of course, if the distinguished normal form is a polynomial or an analytic function, it is uniquely determined.

The following result characterizes the existence of $C^\infty$
distinguished normal forms for $C^\infty$ hyperbolic diffeomorphisms. Part of its proof is well known. In order for our
paper to be self--contained,  we shall provide a complete proof to it.

\begin{lemma}\label{le4} The $C^\infty$ locally hyperbolic
diffeomorphism $F(x)=Ax+f(x)$ is $C^\infty$ conjugate to its distinguished
normal form.
\end{lemma}

\begin{proof} By the assumption given at the beginning of this subsection that $A$ is in the lower triangular normal form,  
$F(x)$ may be nonreal in the case that $A$ has complex conjugate eigenvalues.

Let $x=y+h(y)$ be the conjugation (maybe only
formally) between $F(x)=Ax+f(x)$ and $G(y)=Ay+g(y)$, where
$h(y),g(y)=O(|y|^2)$. Then by definition we have
\begin{equation}\label{e2}
h(Ay)-Ah(y)=f(y+h(y))+h(Ay)-h(Ay+g(y))-g(y).
\end{equation}

Expanding $f,g,h$ in the Taylor series, i.e. for $w\in\{f,g,h\}$ we
have
\[
w(x)\sim\sum\limits_{i=2}\limits^\infty w_i(x).
\]
Then equation \eqref{e2} can be written in
\begin{equation}\label{e2.1}
h_k(Ay)-Ah_k(y)=F_k(y)-g_k(y),\quad k=2,3,\ldots
\end{equation}
where $F_k(y)$, $k=2,3,\ldots$, obtained from re--expansion of
$f(y+h(y))+h(Ay)-h(Ay+g(y))$, are inductively known homogeneous
polynomials of degree $k$. They are functions of $g_j,h_j$ for
$j=2,\ldots,k-1$ and of $f_j$ for $j=2,\ldots,k$.

Define linear operators on $\mathcal H_n^k(\mathbb F)$ by
\[
\mathcal L_kh(y)=h(Ay)-Ah(y)\quad \mbox{ for } h\in\mathcal
H_n^k(\mathbb F).
\]
Separate $F_k(y)-g_k(y)$ in two parts: one is formed by
the nonresonant monomials written in $F_{k1}(y)-g_{k1}(y)$ and the
other is formed by the resonant monomials written in
$F_{k2}(y)-g_{k2}(y)$. Corresponding to the nonresonant part, it
follows from Lemma \ref{le3} that $\mathcal L_k$ is invertible. We
choose $g_{k1}(y)=0$. Equation $\mathcal L_kh_1(y)=F_{k1}(y)$ has
a unique solution in $\mathcal H_n^k(\mathbb F)$, denoted by $h_{k1}(y)$,
which consists of nonresonant monomials. For the resonant part, we
choose $g_{k2}=F_{k2}$ equation $\mathcal L_kh_2(y)=0$ has always
the trivial solution $h_{k2}(y)\equiv 0$. Set
$h(y)=\sum\limits_{k=2}\limits^{\infty}
h_{k}(y)=\sum\limits_{k=2}\limits^{\infty} h_{k1}(y)$, and
$g(y)=\sum\limits_{k=2}\limits^{\infty}
g_{k}(y)=\sum\limits_{k=2}\limits^{\infty} g_{k2}(y)$. Then
$x=y+h(y)$ is the distinguished normalization (maybe formally)
between $F(x)$ and $Ay+g(y)$. This proves that $F(x)$ is conjugate
(maybe formally) to its distinguished normal form $G(y)$.

If $G(y)=Ay+g(y)$ is $C^\infty$ smooth, then by Theorem \ref{le2} $F(x)$
and $G(y)$ are $C^\infty$ conjugate, because $F(x)$ and $G(y) $
both have the origin as a hyperbolic fixed point.

If $G(y)$ is
only a formal series, we get from Lemma \ref{lerr}
that there exists a $C^\infty$ diffeomorphism $H(y)$ such that
$jet_{k=0}^\infty H(y)=G(y)$, where $jet_{k=0}^{\infty}H(y)$
denotes the Taylor series of $H(y)$ at $y=0$. Moreover $H(y)$ has
the origin as a hyperbolic fixed point because $H$ and $G$ have
the same linear part. From the above proof we get that $F(x)$ and
$H(y)$ are formally conjugate. Since $F(x)$ and $H(y)$ have the
origin as a hyperbolic fixed point, it follows from Theorem
\ref{le2} that $F(x)$ is $C^\infty$ conjugate to $H(y)$. Obviously
$H(y)$ is the distinguished normal form of $F(x)$, because $G(y)$
and $H(y)$ have the same Taylor series. We complete the proof of
the lemma.
\end{proof}

From Lemmas \ref{le1} and \ref{le4}, in order for proving $F(x)$
to have a $C^\infty$ embedding flow, it is equivalent to show that
its $C^\infty$ conjugate distinguished normal form has a
$C^\infty$ embedding flow.
We mention that Lemma \ref{le4} will also be used in the proof of Theorem \ref{ex1}.

For a vector field $\mathcal X(x)=Bx+p(x)$ in $(\mathbb R^n,0)$
with $p(x)=O(|x|^2)$, or its associated differential system $ \dot
x=Bx+p(x)$, let $\mu=(\mu_1,\ldots,\mu_n)$ be the $n$--tuple of eigenvalues of the
matrix $B$. A monomial $p_mx^m e_j$ in the 
$j^{th}$ component of $p(x)$, with $m\in
\mathbb Z_+^n$ and $e_j\in\mathbb R^n$ the unit vector having its $j^{th}$ entry being equal
to one, is
\begin{itemize}

\item{} {\it resonant}  if $m$ satisfies $\mu_j=\langle m,\mu\rangle$;
\item{} {\it weakly resonant} if $m$ satisfies $\mu_j-\langle m,\mu\rangle=2l\pi\sqrt{-1}$ for some
$l\in\mathbb Z\setminus\{0\}$;
\end{itemize}

We should mention that for two real $n\times n$ matrix $A$ and
$B$ with $B$ the logarithm of $A$, if a monomial $x^me_j$ is resonant in
a diffeomorphism $Ax+f(x)$, it may be either resonant or weakly resonant in
a vector field $Bx+v(x)$. For example, the matrix $A^*=\mbox{diag}(4,-2,-2)$ has a real
logarithm $B^*=\mbox{diag}(2\ln 2,B_2)$ with $B_2=\left(\begin{array}{cc}\ln2 & \pi\\ -\pi & \ln 2\end{array}
\right)$. If the monomial $x_2x_3e_1$ appears in a diffeomorphism of the form $F^*(x_1,x_2,x_3)=A^*x+f^*(x)$
and also appears in a vector field of the form $\mathcal X^*(x_1,x_2,x_3)=B^*x+v^*(x)$, then it is a resonant
term in both cases. But the monomials $x_2^2e_1$ and $x_3^2e_1$ are resonant in $F^*$ (if appear) and are
only weakly resonant in $\mathcal X^*$ (if appear).

The next result characterizes the form of embedding vector fields for a diffeomorphism
with its nonlinear part consisting of resonant monomials. Part of its proof follows from
that of Theorem 1.1 of \cite{Zh2}.

\begin{lemma}\label{le5}
Assume that $G(y)=Ay+g(y)$ is a $C^\infty$ diffeomorphism and is
in the normal form in $(\mathbb R^n,0)$. If $G(y)$ has a $C^\infty$
embedding autonomous vector field, then the
embedding vector field consists of only resonant and weakly resonant monomials.
\end{lemma}

\begin{proof}
It is well-known that if ${\mathcal X}(y)$ is an embedding vector
field of the diffeomorphism $ G(y)$ in $(\mathbb R^n,0)$ and
${{\varphi}_t(y)}$ is the corresponding embedding flow, then
${\mathcal X}({\varphi}_t(y))=D({\varphi}_t(y)){\mathcal X}(y)$
for $y\in (\mathbb R^n,0)$ and $t\in \mathbb R$, where
$D{\psi}(y)$ denotes the Jacobian matrix of a differentiable map
${\psi}$ at $y$. This implies that
\begin{equation} \label{e3}
{\mathcal X}(G(y))=D(G(y)){\mathcal X}(y),
\end{equation}
because of $G(y)=\varphi_1(y)$ by the assumption. Equation \eqref{e3}
is called the {\it embedding equation} of $G(y)$.

We should say that the vector field $\mathcal X(y)$ satisfying the
embedding equation \eqref{e3} is not necessary an embedding vector
field of $G(y)$. But the set of solutions $\mathcal X(y)$ of
\eqref{e3} provides the only possible candidate embedding vector
fields of the diffeomorphism $G(y)$.

By the assumption of the lemma, the diffeomorphism
${G}(y)$ has an embedding autonomous vector field, we denote it by ${\mathcal X}(y)={B}y+{v}(y)$.
Then it is necessary
that $B=\log A$. Furthermore we get from equation \eqref{e3} that
$v(y)$ satisfies
\begin{equation} \label{e4}
{v}({A}y+{g}(y))-{Av}(y)=D{g}(y){v}(y)+D{g}(y){ B}y-Bg(y).
\end{equation}

Let
\[
g(y)\sim\sum\limits_{j=r}\limits^{\infty}g_j(y),\quad
Dg(y){B}y-Bg(y)\sim-\sum\limits_{j=r}\limits^{\infty}b_j(y)\,\,\,\mbox{
 and }\,\,\,  v(y)\sim\sum\limits_{j=r}\limits^{\infty}v_j(y),
\]
with $r\ge 2$, be the Taylor expansion of the given functions,
where $g_j(y)$, $b_j(y)$ and $v_j(y)$ are the vector--valued
homogeneous polynomials of degree $j$. Equating the homogeneous
polynomials of the same degree of equation \eqref{e4} yields that
$v_j(y)$, $j=r,r+1,\ldots$, should satisfy
\begin{eqnarray}
\mathcal Lv_r(y)=Av_r(y)-v_r(Ay)&=&b_r(y),\label{e5}\\
\mathcal
Lv_k(y)=Av_k(y)-v_k(Ay)&=&-\sum\limits_{j=r}\limits^{k-r+1}Dg_{k+1-j}(y)v_{j}(y)
+ h_k(y)+b_k(y),\label{e6}
\end{eqnarray}
for $k=r+1,r+2,\ldots$, where $h_k$ are the homogeneous
polynomials of degree $k$ in the expansions of $\sum\limits_{r\le
j<k} v_j(Ay+g(y))$ in $y$, and $g_s=0$ if $s<r$. We next prove
that $v_k(y)$, $k\ge r$, must be resonant in the eigenvalues of $A$.

Since $A=(a_{ij})$ is in the lower triangular Jordan normal form, we
assume without loss of generality that $a_{ii}=\lambda_i$ the
eigenvalues of $A$, $a_{i,i-1}=\tau_i=0$ or $1$ where if
$\tau_i=1$ then $\lambda_{i-1}=\lambda_i$  and the other entries
are equal to zero. Then we have
$Ay+g(y)=(\la_1y_1+g_1(y),\la_2y_2+\tau_1y_1+g_2(y),\ldots,\la_ny_n+\tau_{n-1}y_{n-1}+g_n(y))$,
where $g_i(y)$ is the $i^{th}$ component of $g(y)$.

We claim that if $g(y)$ has only resonant terms, then so is the
function $Dg(y){B}y-{Bg}(y)$. Its proof follows from Lemma \ref{le3.1}
(see also Lemma 2.3 of \cite{Zh2}). Especially if $B$ is diagonal in the real Jordan normal form, then
$Dg(y)By-Bg(y)\equiv 0$. Its proof will be provided in the appendix.

This last claim shows that all the vector--valued homogeneous
polynomials $b_j(x)$ for $j=r,r+1,\ldots$ are resonant. So it
follows from Lemma \ref{le3} that the solution $v_r(y)$ of
equation \eqref{e5} should consist of resonant monomials with resonance in the sense of
$A$. Otherwise, corresponding to the nonresonant monomials the operator
$\mathcal L$ is invertible, and so they must vanish.

In what follows we shall prove by induction that the right hand
side of \eqref{e6} consists of the resonant monomials.

By induction we assume that $v_j(x)$  for
$j=r,\ldots,k-1$ are resonant homogeneous
polynomial solutions of degree $j$ of equation \eqref{e6}. Since all monomials $g_i^{({q})}y^{q}$ in $g_i$
are resonant by the assumption, we have ${\la}^{q}=\la_i$. If
$\tau_{i-1}=1$ then $y_{i-1}$ satisfies the same resonant
conditions as those of $y_i$. This shows that all the monomials in the
vector--valued polynomials $h_k$ are resonant. In the sum $\sum$ of 
the right hand side of \eqref{e6} its $i^{th}$ component is the sum
of the polynomials of the form
\begin{equation}\label{rr}
\sum\limits_{s=1}\limits^{n}\frac{\partial g_{k+1-j,i}}{\partial
y_s}(y) v_{j,s}(y).
\end{equation}
Since $g(y)$ contains only resonant terms, it follows from the
induction that the monomials in $\partial g_{k+1-j,i}/\partial
y_s$ are of the form $y^{q}$ modulo the coefficient with $q$
satisfying ${\la}^{q+e_s}=\la_i$, and that the monomials in 
$v_{j,s}$ have the power ${p}$ satisfying ${\la}^{p}=\la_s$. This
implies that the homogeneous polynomial \eqref{rr} contains only
resonant terms. So we have proved that every monomial in the sum on
the right hand side of (\ref{e6}) is resonant, and consequently
all the terms on the right hand side of \eqref{e6} are resonant.

From Lemma \ref{le3} we obtain that the solution $v_k(y)$ of
equation \eqref{e6} with $k>r$ should be a resonant homogeneous
polynomials of degree $k$. Summarizing the above proof we get that
if equation \eqref{e4} has a solution $v(y)$, it should consists
of resonant polynomials, where the resonance is in the sense of the eigenvalues of $A$.

Let $\mu=(\mu_1,\ldots,\mu_n)$ be the $n$--tuple of eigenvalues of $B$. Then we have $\lambda_j=e^{\mu_j}$.
If a monomial $y^me_j$ is resonant in $\lambda$, then $\lambda_j=\lambda^m$. It follows that
$e^{\langle m,\mu\rangle-\mu_j}=1$. Hence we have $\langle m,\mu\rangle-\mu_j=2l\pi\sqrt{-1}$
for some $l\in \mathbb Z$. This implies that the nonlinear part of the vector field $\mathcal X(y)=By+v(y)$ consists
of resonant and weakly resonant monomials. We complete the proof of the lemma.
\end{proof}

\begin{lemma}\label{le6}
Assume that $G(y)=Ay+g(y)$ is a $C^\infty$ locally hyperbolic
diffeomorphism with $g(y)=O(|y|^2)$ and is in the distinguished  normal form in
$(\mathbb F^n,0)$. If $A$ has the logarithm $B$ with its eigenvalues
weakly nonresonant, then $G(y)$ has a $C^\infty$ embedding
flow.
\end{lemma}

\begin{proof} Let ${\mu}=(\mu_1,\ldots,\mu_n)$ be the $n$--tuple of eigenvalues
of $B$. Since $B$ is a logarithm of $A$ and it is in the lower
triangular Jordan normal form, if ${\la}=(\la_1,\ldots,\la_n)$ are
the $n$--tuple of eigenvalues of $A$, then $\la_j=e^{\mu_j}$ for
$j=1,\ldots,n$. If $\la$ are resonant, i.e. there exists some
$j\in\{1,\ldots,n\}$ such that $\la_j={\la}^m$ with $m\in\mathbb
Z_+^n$ and $|m|\ge 2$, then $\mu$ are either resonant or weakly resonant. In the latter, there exists $k\in\mathbb Z$ such that
$\mu_j-\sum\limits_{i=1}\limits^{n} m_i\mu_i=2k\pi\sqrt{-1}$.

Assume that $\mathcal X(y)$ is an embedding vector field of $G(y)$. Taking an
invertible linear change of the coordinates under which $\mathcal X(y)$ is
transformed to a new one with its linear part in the lower triangular
Jordan normal form (if necessary we may use the conjugate complex
coordinates instead of the pairs of real coordinates, which correspond to the conjugate complex eigenvalues). For
simplifying the notations we also use $G(y)$ and $\mathcal X(y)$ to
denote them in the new coordinates, respectively.

Since $G(y)$ is in the distinguished normal form, it follows from Lemma \ref{le5} that $\mathcal X(y)$
consists of resonant and weakly resonant monomials. By the assumption of the lemma
that the eigenvalues of the matrix $B$ is weakly nonresonant, it follows that $\mathcal X(y)$
consists of only resonant monomials.

Let $\varphi (t,y)$ be the flow of $\mathcal X(y)$. Then we
have
\begin{equation}\label{e7}
\dot \varphi(t,y)=\mathcal X(\varphi(t,y)),\quad \varphi(0,y)=y.
\end{equation}
Now we prove the existence of $\mathcal X(y)$ whose flow
$\varphi(t,y)$ satisfies $\varphi(1,y)=G(y)$. Part of the ideas
follows from \cite{CLW} on the proof of the Takens' theorem, where
they obtained the existence of a finite jet of an embedding flow
for a finite jet of a diffeomorphism. Taking the Taylor expansions
of $\mathcal X(y)$ and of $\varphi(t,y)$ as
\[
\mathcal X\sim  By+X_2(y)+X_3(y)+\ldots \mbox{ and } \varphi(t,y)\sim
\varphi_1(t,y)+\varphi_2(t,y)+\varphi_3(t,y)+\ldots.
\]
Then we get from equation
\eqref{e7} that
\begin{eqnarray}\label{e8}
&& \dot\varphi_1(t,y)=B\varphi_1(t,y),\quad\hspace{45mm} \varphi_1(0,y)=y,\\
&&\dot\varphi_k(t,y)=B\varphi_k(t,y)+X_k(\varphi_1(t,y))+P_k(t,y),\quad
\varphi_k(0,y)=0,\label{e9}
\end{eqnarray}
for $k=2,3,\ldots$, where $\varphi_k$ and $X_k$ are unknown, $P_2=0$ and $P_k$ for $k>2$ are
inductively known vector--valued homogeneous polynomials of degree $k$ in $y$ and
it is a polynomial in $\varphi_2,\ldots,\varphi_{k-1}$ and
$X_2,\ldots,X_{k-1}$. In fact, $P_k$ is obtained from the expansion of $X_2(\phi)+
\ldots+X_{k-1}(\phi)$.

Equation \eqref{e8} has the solution $\varphi_1(t,y)=e^{tB}y$. Its
time one map is $e^By=Ay$ by the assumption. Now we shall prove by
induction that for $k=2,3,\ldots,$ equation \eqref{e9} has a
homogeneous polynomial solution of degree $k$ in $y$ whose time
one map is $g_r(y)$, where $g_r(y)$ is the homogeneous polynomial of degree
$r$ in the Taylor expansion of $G(y)$.

Be induction we assume that for $k=2,\ldots,r-1$, system \eqref{e9} has a
homogeneous polynomial solution $\varphi_k(t,y)$ of degree $k$ in
$y$ satisfying $\varphi_k(1,y)=g_k(y)$. For $k=r$, $P_k(t,y)$ is
known, $\varphi_k(t,y)$ and $X_k(y)$ are both unknown. By the
variation of constants formula, we get from \eqref{e9} with $k=r$
that
\[
\varphi_r(t,y)=e^{tB}\int\limits_0\limits^te^{-sB}(X_r(e^{sB}y)+P_r(s,y))ds.
\]
In order that $\varphi(t,y)$ is the embedding flow of $G(y)$, we should
have $\varphi(1,y)=G(y)$. This means that $\varphi_r(1,y)=g_r(y)$, so
we should have
\begin{equation}\label{e10}
\int\limits_0\limits^1e^{-sB}X_r\left(e^{sB}y\right)ds=e^{-B}g_r(y)-\int\limits_0\limits^1e^{-sB}P_r(s,y)ds.
\end{equation}

Define $\mathcal R^r$ to be the subspace of $\mathcal
H_n^r(\mathbb F)$ which consists of the vector--valued resonant
homogeneous polynomials of degree $r$, where resonance is in the sense of eigenvalues of
$B$. From Lemma \ref{le5} we
know that if the solution of equation \eqref{e10} exists, it
should be in $\mathcal R^r$. For proving the existence of the
solution in $\mathcal R^r$ of equation \eqref{e10}, we define the
operator
\begin{eqnarray}\label{rrr}
T^r:\, \mathcal R^r & \longrightarrow & \mathcal R^r\\
X_r&\longrightarrow & \int\limits_0\limits^1e^{-sB}X_r(e^{sB}y)ds.\nonumber
\end{eqnarray}
Obviously $T^r$ is linear in $\mathcal R^r$. Firstly we claim that $T^r(X_r)\in\mathcal R^r$ for each
$X_r\in\mathcal R^r$. Indeed, since $B$ is in the Jordan normal form and
$B=\mbox{diag}(J_1,\ldots,J_m)$ with $J_i$ the $n_i$--th lower triangular Jordan block.
Set $z_j=(y_{n_{j-1}+1},\ldots,y_{n_j})$ for $j=1,\ldots,m$, where $n_0=0$.
Then
$
e^{sB}y=\left(e^{sJ_1}z_1, \ldots, e^{sJ_m}z_m\right)^T,
$ where $T$ denotes the transpose of a matrix.
For each $z_j$, $j=1,\ldots,m$, its all components correspond to the same resonant condition.
So if $X_r(x)$ is resonant as a function of $x$ then $X_r(e^{sB}y)$ is resonant
as a function of $y$. Moreover, using the block diagonal form of $e^{-sB}$ we get
that $T^r(X_r(y))$ consists of the resonant monomials. This proves the claim.

Secondly we claim that $T^r$ is invertible in $\mathcal R^r$.
For proving this claim, we need to compute the expression $T^r(y^m e_j)$
for each base element $y^m e_j$ of $\mathcal
R^r$ with $j=1,\ldots,n$, $m\in\mathbb Z_+^n$ and $|m|=r$. Recall that $e_j$ is the $j^{th}$ unit vector
with its $j^{th}$ entry being equal to $1$ and the others vanishing.

Since $B$ is in the lower triangular Jordan normal form, we assume
without loss of generality that $B=S+N$ and $SN=NS$, where $S$ is
semisimple, and $N$ is nilpotent and is in the lower triangular normal form.
Furthermore we can assume that $S$ is in the diagonal form, i.e.
$S=\mbox{diag}(\mu_1,\ldots,\mu_n)$.

For $X_r(y)=y^m e_j$ with $|m|=r$, since
\[
\left(e^{sB}y\right)^m e_j=\left(e^{sS}e^{sN}y\right)^m e_j
\quad \mbox{ and } \quad e^{sN}=\left(\begin{array}{ccc}1 & & 0 \\ & \ddots &\\
* & & 1\end{array}\right),
\]
where $*$ are the entries consisting of polynomials in $s$, we
have
\[
\left(e^{sB}y\right)^m e_j=e^{\langle m,\mu\rangle s}y^m
e_j+q_r(s,y)e_j,
\]
where $q_r(s,y)e_j$ is a polynomial consisting of monomials which
are before $y^m e_j$ in the lexicographic ordering. Recall that a
monomial $y^\sigma$ is before $y^\alpha$ in the lexicographic
ordering if there exists an $l$ with $1\le l\le n$ such that
$\sigma_j=\alpha_j$ for $0\le j<l$ and $\sigma_l>\alpha_l$. It
follows that for $X_r(y)=y^me_j$  with $|m|=r$
\begin{equation}\label{rrr1}
e^{-sB}X_r\left(e^{sB}y\right)=e^{-sB}\left(e^{sB}y\right)^m
e_j=e^{(\langle m,\mu\rangle-\mu_j)s}y^m e_j+Q_r(s,y)e_j+\sum\limits_{k=j+1}\limits^nP_k(s,y)e_k,
\end{equation}
where $Q_r(s,y)$ is a homogeneous polynomial in $y$ of degree $r$
and is before $e^{(\langle m,\mu\rangle-\mu_j)s}y^m$, and $P_k(s,y)$ are also homogeneous polynomials in $y$ of degree $r$ for $k=j+1,\ldots,n$. Since $y^m e_j$ is a resonant monomial, we have $\langle
m,\mu\rangle-\mu_j=0$. This shows that
\[
T^r(y^m e_j)= \int\limits_0\limits^1e^{-sB}(e^{sB}y)^m e_jds=y^m
e_j+ \int\limits_0\limits^1Q_r(s,y)dse_j+\sum\limits_{k=j+1}\limits^n\int\limits_0\limits^1 P_k(s,y)dse_k\ne 0.
\]
This last expression implies that the matrix expression of $T^r$ under the basis
$\{y^me_j;\, m\in\mathbb Z_+^n, |m|=r,j=1,\ldots,n\}$
is a lower triangular matrix and its diagonal entries are all equal to $1$.
So $T^r$ is invertible on $\mathcal R^r$. The claim follows.

These last two claims mean that $T^r$ is an invertible linear operator
in $\mathcal R^r$. Furthermore, by induction and working in a similar way
to the proof of $T^r(X_r)$ belonging to $\mathcal R^r$ we can prove
that $P_r(s,y)$ is resonant as a function of $y$. So the right hand side of
\eqref{e10} belongs to $\mathcal R^r$. This shows that equation \eqref{e10} has a unique
resonant homogeneous polynomial solution $X_r$ in $\mathcal R^r$.

By induction we have proved that each equation in \eqref{e9} for
$k=2,3,\ldots$ has a unique resonant homogeneous polynomial
solution $X_k$ in $\mathcal R^k$. Consequently, the given
$C^\infty$ diffeomorphism $G(y)$ has a unique (maybe formal)
autonomous embedding vector field $\mathcal
X(y)=By+\sum\limits_{k=2}\limits^{\infty}X_r(y)$.

By Lemma \ref{lerr} there exists a $C^\infty$ vector field $\mathcal
Y(y)$ such that $jet_{k=1}^{\infty}\mathcal
Y(y)=By+\sum\limits_{k=2}\limits^{\infty}X_r(y)$. Consider the
$C^\infty$ vector field $\mathcal Y(y)$, and denote by $\psi(y)$
its time one map of the flow associated with $\mathcal Y(y)$. Then
we have $jet_{k=1}^{\infty}\psi(y)=jet_{k=1}^{\infty}G(y)$, and
consequently they are formally conjugate. Since $\psi(y)$ and
$G(y)$ are hyperbolic at $y=0$, it follows from Theorem \ref{le2}
that $\psi(y)$ and $G(y)$ are $C^\infty$ conjugate. Since
$\mathcal Y(y)$ is the embedding vector field of $\psi(y)$, we get
from Lemma \ref{le1} that $G(y)$ has a $C^\infty$ embedding vector
field. This proves the lemma. \end{proof}

\noindent{\it Proof of Theorem }\ref{th1}: Its proof follows from
Lemmas \ref{le4}, \ref{le6} and \ref{le1}. More details, Lemma
\ref{le4} shows that the given $C^\infty$ hyperbolic
diffeomorphism $F(x)$ is $C^\infty$ conjugate to its distinguished
normal form, denote by $G(y)$. By Lemma \ref{le6} the
diffeomorphism $G(y)$ has a $C^\infty$ embedding autonomous  vector
field. Finally it follows from Lemma \ref{le1} that $F(x)$ has a
$C^\infty$ embedding autonomous vector field, and consequently has
a $C^\infty$ embedding flow. We complete the proof of the theorem.
$\bbox$

We remark that in the proof of Lemma \ref{le6}, if $\mathcal R^r$ consists of both resonant and weakly
resonant vector--valued monomials of $\mathcal H_n^r(\mathbb F)$, then the linear operator
$T^r$ defined in \eqref{rrr} is not invertible. Because for those weakly resonant monomials
$y^me_j$ we get from \eqref{rrr1} that
$
T^r(y^m e_j)=\int\limits_0\limits^1Q_r(s,y)e_j+\sum\limits_{k=j+1}\limits^n\int\limits_0\limits^1 P_k(s,y)dse_k$, where $Q_r(s,y)$
consists of homogeneous polynomials in $y$ before $y^m $ and $P_k(s,y)$ are also polynomials in $y$. This implies that the
matrix expression of $T^r$
under the basis $\{y^me_j;\, m\in\mathbb Z_+^n,|m|=r,j=1,\ldots,n\}$ is in the lower triangular form
and its diagonal entries have both $1$ and $0$. This explains from a new point
the importance of the weakly nonresonant condition for ensuring the existence of smooth embedding flows
for a smooth diffeomorphism.

\subsection{\normalsize Proof of Theorem \ref{ex1}}\label{s23}

For proving the theorem we provide a class of $C^\infty$ locally hyperbolic
diffeomorphisms with their linear part having a real logarithm  whose eigenvalues
are weakly resonant, but they do not have a $C^\infty$ embedding autonomous
vector field.

Consider a $C^\infty$ or an analytic real local diffeomorphism of the form
\begin{equation}\label{eex}
F(x)=Ax+f(x),\qquad x\in (\mathbb R^n,0),
\end{equation}
with $f(x)=o(x)$ and $A=\mbox{diag}(A_1,A_2)$, where
\[
A_1=\left(\begin{array}{ccc}e^8 & 0 & 0\\
0 & \frac{\sqrt{2}}{2}e & \frac{\sqrt{2}}{2}e\\ 0 & -\frac{\sqrt{2}}{2}e & \frac{\sqrt{2}}{2}e
\end{array}\right),
\]
and $A_2=\mbox{diag}(\lambda_4,\ldots,\lambda_n)$ with its eigenvalues
$\lambda_i>1$ for $i=4,\ldots,n$. We assume that $\lambda_4,\ldots,\lambda_n$ are nonresonant, and that they are not
resonant with the eigenvalues of $A_1$. The matrix $A_2$ with the prescribed
property can be easily illustrated.
Obviously, $A_1$ has the eigenvalues $(\lambda_1,\lambda_2,\lambda_3)=(e^8,e^{1+\frac{\pi}{4}\sqrt{-1}},
e^{1-\frac{\pi}{4}\sqrt{-1}})$. From the choice of $A_2$ we know that the eigenvalues of $A$ satisfy
only the resonant relations $\lambda_1=\lambda_2^4\lambda_3^4$, $\lambda_1=\lambda_2^8$ and $\lambda_1=\lambda_3^8$.

From the selection of $A$ and Proposition \ref{po1}, the matrix $A$ has a real logarithm $B=\mbox{diag}(B_1,B_2)$
with
\[
B_1=\left(\begin{array}{ccc}8 & 0 & 0\\
0 & 1 & \frac{\pi}{4}\\ 0 & -\frac{\pi}{4} & 1
\end{array}\right),
\]
and $B_2=\mbox{diag}(\ln\lambda_3,\ldots,\ln\lambda_n)$. Moreover the resonant
and weakly resonant relations associated with the eigenvalues of $B$ are only
\[
\mu_1=4\mu_2+4\mu_3,\quad \quad \mu_1-8\mu_2=-2\pi\sqrt{-1} \quad \mbox{and }
\quad \mu_1-8\mu_3=2\pi\sqrt{-1},
\]
where
$(\mu_1,\mu_2,\mu_3)=(8, 1+\frac{\pi}{4}\sqrt{-1},1-\frac{\pi}{4}\sqrt{-1})$ are the eigenvalues of $B_1$.

We can check easily that the eigenvalues of $A$ belong to the Poincar\'e domain and that
$F(x)$ is locally hyperbolic.
Instead of $x_2,x_3$ we use the conjugate complex
coordinates, and also denote them by $x_2,x_3$. Then by Lemma \ref{le4} (resp. Theorem \ref{le7})
we get that if the diffeomorphism $F(x)$ is $C^\infty$ (resp. analytic),
it is $C^\infty$ (resp. analytically) conjugate to the local diffeomorphism
of the form
\[
G(x)=\left(e^8x_1+ax_2^4x_3^4+b x_2^8+d x_3^8,\,e^{1+\frac{\pi}{4}\sqrt{-1}}x_2,\,
e^{1-\frac{\pi}{4}\sqrt{-1}}x_3,\,\lambda_4 x_4,\,\ldots,\,\lambda_nx_n\right),
\]
with $a\in\mathbb R$ and $\overline b=d\in \mathbb C$, where $\overline b=d$ follows form the fact that written in real coordinates
the first component of $G(x)$ should be real.
Hence by Lemma \ref{le1}, $F(x)$ has a $C^\infty$ (resp. analytic)
embedding autonomous vector field if and only if $G(x)$ has a $C^\infty$ (resp. analytic)
embedding autonomous vector field. So we turn to study the existence of embedding autonomous
vector field for the local diffeomorphism $G(x)$.

Since $G(x)$ contains only resonant nonlinear terms, it follows from Lemma \ref{le5}
that if the embedding autonomous vector field of $G(x)$ exists,
it should be of the form
\[
\mathcal X(x)=\left(8x_1,(1+\frac{\pi}{4}\sqrt{-1})x_2, (1-\frac{\pi}{4}\sqrt{-1})x_3,
\ln\lambda_4x_4,\ldots,\ln\lambda_nx_n\right)+g(x),
\]
with $g(x)$ consisting of resonant and weakly resonant monomials.
So the candidate $g(x)$ is only possible of the form
\[
g(x)=\left(Ax_2^4x_3^4+Bx_2^8+Dx_3^8,0,0,0,\ldots,0\right).
\]

Some easy computations show that the flow of $\mathcal X(x)$ is
\[
\phi_t(x)=\left(\begin{array}{c}
e^{8t}\left[x_1+Ax_2^4x_3^4t+\frac{B}{2\pi\sqrt{-1}}\left(e^{2\pi\sqrt{-1}t}-1\right)x_2^8-
\frac{D}{2\pi\sqrt{-1}}\left(e^{-2\pi\sqrt{-1}t}-1\right)x_3^8\right]\\
e^{(1+\frac{\pi}{4}\sqrt{-1})t}x_2\\
e^{(1-\frac{\pi}{4}\sqrt{-1})t}x_3\\
\lambda_4^tx_4\\
\vdots\\
\lambda_n^tx_n
\end{array}
\right).
\]
Clearly, the time one map of $\phi_t$ is $\left(
e^{8}(x_1+Ax_2^4x_3^4), e^{(1+\frac{\pi}{4}\sqrt{-1})}x_2, e^{(1-\frac{\pi}{4}\sqrt{-1})}x_3,\lambda_4x_4,
\ldots,\lambda_nx_n\right)$. So,
if $\overline b=d\ne 0$, the diffeomorphism $G(x)$ cannot have a $C^\infty$ or an analytic embedding autonomous vector field.
Consequently $F(x)$ cannot have a $C^\infty$ or an analytic embedding autonomous vector field. This proves the
theorem.

\subsection{\normalsize Proof of Theorem \ref{co1}}\label{s24}

$(a)$ Necessity. By the assumption $F(x)$ has an embedding autonomous vector field, denoted by
$\mathcal X(x)=Bx+v(x)$ with $v(x)$ the higher order terms. Then $B$ is the real logarithm of $A$. It follows from Proposition
\ref{po1} that either $A$ has no negative eigenvalues or the Jordan blocks corresponding to the negative eigenvalues
of $A$ appear pairwise. Since we are in the two dimensional space,  the later case means that the Jordan normal form should be of
the form $\mbox{diag}(-\lambda,-\lambda)$ with $\lambda>0$. This proves the necessary part.

\noindent Sufficiency. If $A$ has no negative eigenvalues, then $A$ has the Jordan normal form of the type either
\[
J_1=\left(\begin{array}{cc}\lambda_1 & 0\\ 0 & \lambda_2 \end{array}\right),\quad \mbox{or }\quad
J_2=\left(\begin{array}{cc}\lambda_1 & 0\\ 1 & \lambda_1 \end{array}\right),
\quad \mbox{or }\quad
J_3=\left(\begin{array}{cc}\alpha & \beta\\ -\beta & \alpha \end{array}\right),
\]
with $\lambda_1,\lambda_2>0$ and $\alpha\beta\ne 0$. They have respectively the real logarithms
\[
\ln J_1=\left(\begin{array}{cc} \ln \lambda_1 & 0 \\ 0 & \ln\lambda_2\end{array}\right),\quad
\ln J_2=\left(\begin{array}{cc} \ln \lambda_1 & 0 \\ \lambda_1^{-1} & \ln\lambda_1\end{array}\right),\]
\[
\ln J_3=\left(\begin{array}{cc} \frac 12\ln (\alpha^2+\beta^2) & \arccos\frac{\alpha}{\sqrt{\alpha^2+
\beta^2}} \\ -\arccos\frac{\alpha}{\sqrt{\alpha^2+\beta^2}} & \frac 12\ln (\alpha^2+\beta^2)\end{array}\right).
\]
In any one of the above three cases the eigenvalues are not possible weakly resonant.

If $A$ has negative eigenvalues, by the assumption the Jordan normal form of $A$
can only be of the form $J_4=\mbox{diag}(-\lambda,-\lambda)$ with $\lambda>0$. It is
easy to check that $J_4$ has the real logarithm
\[
\ln J_4= \left(\begin{array}{cc} \ln \lambda & \pi \\ -\pi & \ln\lambda\end{array}\right).
\]
Obviously, its eigenvalues are not weakly resonant.

The above proof shows that under the assumption of the sufficiency the matrix $A$ has a
real logarithm with its eigenvalues weakly nonresonant. So it follows from Theorem \ref{th1}
that the diffeomorphism $F(x)$ has a local $C^\infty$ embedding flow. This proves the statement.

$(b)$ For the diffeomorphism $F(x)=Ax+f(x)$ with $A$ having only
positive eigenvalues,  it follows from Proposition \ref{po1} that $A$ has a real logarithm.
We claim that the real logarithm can be chosen such that its eigenvalues are weakly
nonresonant. Then statement $(b)$ follows from Theorem \ref{th1}.

We now prove the claim. Let $J$ be the Jordan normal form
of $A$ and $T$ be the nonsingular real matrix such that $A=TJT^{-1}$.
By the assumption of the theorem we can assume that
\[
J=\mbox{diag}(C_1,\ldots,C_p),
\]
with
\[
C_i=\gamma_iE_{n_i}+N_{n_i},
\quad i=1,\ldots,n,
\]
where $n_1+\ldots+n_p=n$ and $E_{n_i}$ is the $n_i$--th unit matrix,
and
\[
N_{n_i}=\left(\begin{array}{cccc}
0 & & & \\
1 & 0 & & \\
 & \ddots & \ddots & \\
  & & 1 & 0  \end{array}\right).
\]
Clearly, $C_i$ has the $n_i$--tuple of eigenvalues $(\gamma_i,\ldots,\gamma_i)$ with $\gamma_i>0$ by the assumption.

We know that $C_i$ has the real logarithm
\[
\ln C_i=(\ln \gamma_i)E_{n_i}+\sum\limits_{k=1}\limits^{\infty}(-1)^{k+1}\frac{1}{k}\left(\gamma_i^{-1}N_{n_i}\right)^k,
\]
with the $n_i$--tuple of real eigenvalues $(\ln\gamma_i,\ldots,\ln\gamma_i)$.
Set $B=T(\ln C_1,\ldots,\ln C_p) T^{-1}$. Then
$A=e^B$, i.e. $B$ is the real logarithm of $A$.

Since the eigenvalues of $B$ are all real, it is not possible weakly resonant. This proves the claim, and
consequently the statement $(b )$. We complete the proof of
the theorem.

\section{Appendix}

\begin{lemma}\label{app}
Let $A$ be a real square matrix of order $n$. Assume that $A$ has a real logarithm $B$. If $B$ is diagonal in the real Jordan normal form and the $n$ dimensional vector function $g(y)$ contains only resonant terms in the sense of diffeomorphism, then
$Dg(y)By-Bg(y)\equiv 0$.
\end{lemma}

\begin{proof}
In the real Jordan normal form we have $A=\mbox{diag}(A_1,\ldots, A_r, B_1,\ldots,B_s,C_1,\ldots,C_k)$ with
\[
A_m=\left(\begin{array}{cccc}\lambda_m & & & \\1 & \lambda_m & & \\ & \ddots &\ddots &\\ & & 1 & \lambda_m\end{array}\right),\,\, \lambda_m>0;\quad
B_m=\left(\begin{array}{cccc}\mu_m & & & \\1 & \mu_m & & \\ & \ddots &\ddots &\\ & & 1 & \mu_m\end{array}\right),\,\, \mu_m<0,
\]
\[
C_m=\left(\begin{array}{cccc}D_m & & & \\E_2 & D_m & & \\ & \ddots &\ddots &\\ & & E_2 & D_m\end{array}\right), D_m=\left(\begin{array}{cc}\alpha_m & \beta_m \\-\beta_m & \alpha_m\end{array}\right),
E_2=\left(\begin{array}{cc}1 & 0 \\ 0 & 1\end{array}\right).
\]
Then $B=\ln A=\mbox{diag}(\ln A_1,\ldots, \ln A_r, \ln B_1,\ldots,\ln B_s,\ln C_1,\ldots,\ln C_k)$. By the assumption that $A$ has a real logarithm, $B_m$ should appear  pairwise. Let $M_m=\mbox{diag}(B_m,B_m)$, then $M_m$ is similar to $\Lambda+Z$ with $\Lambda=\mbox{diag}(\mu_mE_2,\ldots,\mu_mE_2)$ and
\[
Z=\left(\begin{array}{cccc}\bf 0 & & & \\E_2 & \bf 0 & & \\ & \ddots &\ddots &\\ & & E_2 & \bf 0\end{array}\right).
\]
We know that
\[
\ln (\mu_mE_n)=\left(\begin{array}{cc}\ln|\mu_m| & (2k+1)\pi \\ -(2k+1)\pi & \ln|\mu_m|\end{array}\right),\quad k\in\mathbb Z.
\]
In addition,
\[
\ln D_m=\left(\begin{array}{cc}\displaystyle\frac 12\ln(\alpha_m^2+\beta_m^2) & \arccos\frac{\alpha_m}{\sqrt{\alpha_m^2+\beta_m^2}}+2l\pi \\ - \arccos\frac{\alpha_m}{\sqrt{\alpha_m^2+\beta_m^2}}-2l\pi &\displaystyle \frac 12\ln(\alpha_m^2+\beta_m^2) \end{array}\right),\quad l\in\mathbb Z.
\]
So in order that $B=\ln A$ is diagonal in the real Jordan normal form, the necessary condition is that the blocks $B_m$ for $m\in\{1,\ldots,s\}$ and $C_m$  for $m\in\{1,\ldots,k\}$ do not appear in $A$.
These show that $A$ is similar to $\mbox{diag}(\lambda_1,\dots,\lambda_n)$ and and $B$ is similar to $\mbox{diag}(\ln\lambda_1,\ldots,\ln\lambda_n)$. Without loss of generality we assume that $B=\mbox{diag}(\ln\lambda_1,\ldots,\ln\lambda_n)$.

Set $\mu_s=\ln \lambda_s$ for $s=1,\ldots,n$. Now the $s^{th}$ component of $Dg(y)By-Bg(y)$ is
$\sum\limits_{j=1}\limits^n\frac{\partial g_s}{\partial y_j}\mu_jy_j-\mu_sg_s$, where $g_s$ is the $s^{th}$ component of $g$. Since $g_s$ is resonant, for illustration choosing 
one of monomials of $g_s$, saying $y^m=y_1^{m_1}\ldots y_n^{m_n}$, we have
\[
\sum\limits_{j=1}\limits^n\frac{\partial y^m}{\partial y_j}\mu_jy_j-\mu_sy^m=(\mu_1m_1+\ldots+\mu_nm_n-\mu_s)y^m=\left(\ln \frac{\lambda^m}{\lambda_s}\right)y^m=0,
\]
because $y^m$ is resonant and so $\lambda^m=\lambda_s$. This proves that $\sum\limits_{j=1}\limits^n\frac{\partial g_s}{\partial y_j}\mu_jy_j-\mu_sg_s=0$, and consequently $Dg(y)By-Bg(y)=0$. The lemma follows.
\end{proof}

\noindent{\bf Acknowledgements.} The author sincerely appreciates the referee for his/her careful reading and the excellent comments and suggestions,
which help me to improve this paper both in the mathematics and in the expressions. Also the referee pointed out
the existence of the reference \cite{Cu}, which I did not know previously.


\bigskip
\bigskip

\begin{thebibliography}{99}

\bibitem{AA} D. V. Anosov and V. I. Arnold, {\it Dynamical Systems I}, Springer-Verlag, Berlin, 1988.

\bibitem{Ar} V. I. Arnold, {\it Geometric Methods in Theory of Ordinary
Differential Equations}, 2nd Ed., Springer-Verlag, New York, 1988.

\bibitem{Be} G. R. Belitskii and V. Tkachenko, {\it One-Dimensional Functional
Equations}, Birkh\"{a}user, Berlin, 2003.

\bibitem{BC} W. A. Beyer and P. J. Channell,
A functional equation for the embedding of a homeomorphim of the
interval into a flow, {\it Lecture Notes in Math.}, Vol. {\bf
1163}, Springer-Verlag, New York, 1985, 1--13.

\bibitem{Bi} Yu. N. Bibikov,
Local Theory of Nonlinear Analytic Ordinary Differential
Equations, {\it Lecture Notes in Math.}, Vol. {\bf 702},
Springer-Verlag, Berlin, 1979.

\bibitem{Ch} K. T. Chen,
Equivalence and decomposition of vector fields about an elementary
critical point, {\it Amer. J. Math.} {\bf 85} (1963), 693--772.

\bibitem{CLW} S. N. Chow, Chengzhi Li and Duo Wang,
{\it Normal Forms and Bifurcation of Planar Vector Fields},
Cambridge University Press, Cambridge, 1994.

\bibitem{Cu} W. J. Culver,
On the existence and uniqueness of the real logarithm of a matrix, {\it Proc. Amer. Math. Soc.} {\bf 17} (1966), 1146--1151.

\bibitem{EP09} A. Enciso and D. Peralta--Salas, Existence and vanishing
set of inverse integrating factors for analytic vector fields,
{\it Bull. London Math. Soc.} {\bf 41} (2009), 1112--1124.

\bibitem{GGG09} I. A. Garc\'ia, H. Giacomini and M. Grau, The inverse integrating
factor and the Poincar\'e map, {\it Trans. Amer. Math. Soc.} {\bf 362} (2010), 3591--1612.

\bibitem{GM09} I. A. Garc\'ia and S. Maza, A new approach to center
conditions for simple analytic monodromic singularities, {\it J.
Differential Equations} {\bf 248} (2010), 363--380.

\bibitem{KP}  S. Kuksin and J. P$\rm\ddot{o}$schel,
On the inclusion of analytic symplectic maps in analytic
Hamiltonian flows and its applications, {\it Seminar on Dynamical
Systems}, S. Kuksin, V. Lazutkin and J. P$\rm\ddot{o}$schel (Eds),
Birkh$\rm\ddot{a}$user, Basel, 1978, 96--116.

\bibitem{La1} P. F. Lam,
Embedding a differential homeomorphim in a flow, {\it J.
Differential Equations} {\bf 30} (1978), 31--40.

\bibitem{La2}  P. F. Lam,
Embedding homeomorphims in $C^1-$flows, {\it Ann. Math. Pura
Appl.} {\bf 123} (1980), 11--25.

\bibitem{Li} Weigu Li,
{\it Normal Form Theory and its Applications} (in Chinese),
Science Press, Beijing, 2000.

\bibitem{LLZ} Weigu Li, J. Llibre and Xiang Zhang,
Extension of floquet's theory to nonlinear periodic differential
systems and embedding diffeomorphisms in differential flows, {\it
American J. Math.} {\bf 124} (2002), 107--127.

\bibitem{Pa} J. Palis, Vector fields generate few diffeomorphisms,
{\it Bull. Amer. Math. Soc.} {\bf 80} (1974), 503--505.


\bibitem{Zh1} Xiang Zhang, Analytic normalization of analytic integrable systems
and the embedding flows,  {\it J. Differential Equations} {\bf
244} (2008), 1080--1092.

\bibitem{Zh} Xiang Zhang, Embedding diffeomorphisms in flows in Banach spaces,
{\it Ergod. Th. $\&$ Dynam. Sys.} {\bf 29} (2009),  1349--1367.

\bibitem{Zh2} Xiang Zhang, Embedding smooth diffeomorphisms in flows, {\it J. Differential Equations}
{\bf 248} (2010), 1603--1616.

\bibitem{Zh3} Xiang Zhang, Analytic integrable
systems: analytic normalization and embedding flows, preprint.







\end{thebibliography}
\end{document}